\documentclass[11pt]{amsart}

\usepackage{amsmath,amssymb,latexsym,enumerate}

\makeatletter
\def\subsection{\@startsection{subsection}{2}%
  \z@{.5\linespacing\@plus.7\linespacing}{-.5em}%
  {\normalfont\bfseries\boldmath}}
\makeatother

\usepackage[svgnames]{xcolor}

\let\ve=\mathbf
\renewcommand{\c}{\mathbf{c}}
\renewcommand{\a}{\mathbf{a}}
\renewcommand{\b}{\mathbf{b}}

\newcommand{\LL}{\mathcal{L}}

\newcommand{\integers}{\mathbb{Z}}
\newcommand{\rationals}{\mathbb{Q}}
\newcommand{\reals}{\mathbb{R}}
\newcommand{\CC}{\mathcal{C}}
\newcommand{\KK}{\mathcal{K}}

\newcommand{\RR}{\mathcal{R}}
\newcommand{\PP}{\mathcal{P}}

\newcommand{\s}{\mathbf{s}}
\renewcommand{\u}{\mathbf{u}}

\newcommand{\la}{\lambda}

\DeclareMathOperator\lcm{lcm}

\newcommand{\todo}[1]{\par \smallskip\noindent
  \framebox{\begin{minipage}[c]{0.95 \textwidth} \ttfamily \sloppy TO DO:
      #1 \end{minipage}}\par\smallskip}

\newcommand{\note}[1]{\par \smallskip\noindent
  \framebox{\begin{minipage}[c]{0.95 \textwidth} \ttfamily NOTE:
      #1 \end{minipage}}\par\smallskip}

\newcommand{\commentout}[1]{}

\newtheorem{theorem}{Theorem}[section]
\newtheorem{lemma}[theorem]{Lemma}
\newtheorem{proposition}[theorem]{Proposition}
\newtheorem{corollary}[theorem]{Corollary}
\newtheorem{conjecture}[theorem]{Conjecture}

\theoremstyle{definition}
\newtheorem{definition}[theorem]{Definition}

\theoremstyle{remark}

\newtheorem{remark}[theorem]{Remark}

\title[On $\s$-Lecture Hall Partitions]{$\s$-Lecture Hall Partitions, Self-Reciprocal Polynomials, and Gorenstein Cones}

\author[M. Beck]{
Matthias Beck
}
\address{Department of Mathematics\\
         San Francisco State University\\
         San Francisco, CA 94132}
\email{mattbeck@sfsu.edu}

\author[B. Braun]{
Benjamin Braun
}
\address{Department of Mathematics\\
         University of Kentucky\\
         Lexington, KY 40506--0027}
\email{benjamin.braun@uky.edu}

\author[M. K\"oppe]{
Matthias K\"oppe
}
\address{Department of Mathematics\\
         University of California, Davis\\
         One Shields Avenue\\
         Davis, CA 95616}
\email{mkoeppe@math.ucdavis.edu}

\author[C. Savage]{
Carla D. Savage
}
\address{Department of Computer Science\\
         North Carolina State University \\
         Raleigh, NC 27695-8206}
\email{savage@ncsu.edu}

\author[Z. Zafeirakopoulos]{
 Zafeirakis Zafeirakopoulos
}
\address{Research Institute for Symbolic Computation\\
Johannes Kepler University\\
Altenberger Strasse 69\\
A-4040 Linz, Austria}
\email{zafeirakopoulos@risc.jku.at}

\thanks{
Matthias Beck was partially supported by grant DMS-1162638 of the U.S.\ National Science Foundation.
Benjamin Braun was partially supported by grant H98230-13-1-0240 of the U.S.\ National Security Agency.
Matthias K\"oppe was partially supported by grant DMS-0914873 of the U.S.\ National Science Foundation.
The work was partially supported by a grant from the Simons Foundation (\# 244963 to Carla Savage).
Zafeirakis Zafeirakopoulos was supported by the strategic program ``Innovatives O\"O 2010 plus'' by the Upper Austrian Government and by the Austrian Science Fund (FWF): W1214-N15, project DK6.
The authors thank the American Institute of Mathematics for support of our SQuaRE working group on ``Polyhedral Geometry and Partition Theory.'' 
We are deeply grateful to a referee who took time and care and did a thorough job of checking the paper. In the end (s)he understood the paper better than we did.
}

\subjclass[2010]{05A17, 05A19, 52B11, 13A02, 13H10 
}

\begin{document}

\begin{abstract}
In 1997, Bousquet--M\'elou and Eriksson initiated the study of \emph{lecture hall partitions}, a fascinating family of partitions that yield a finite version of Euler's celebrated odd/distinct partition theorem.
In subsequent work on $\s$-lecture hall partitions, they considered the \emph{self-reciprocal property} for various associated generating functions, with the goal of characterizing those sequences $\s$ that give rise to generating functions of the form $((1-q^{e_1})(1-q^{e_2}) \cdots (1-q^{e_n}))^{-1}$.

We continue this line of investigation, connecting their work to the more general context of Gorenstein cones.
We focus on the Gorenstein condition for $\s$-lecture hall cones when $\s$ is a positive integer sequence generated by a second-order homogeneous linear recurrence with initial values $0$ and $1$.
Among such sequences $\s$, we prove that the $n$-dimensional $\s$-lecture hall cone is Gorenstein for all $n \geq 1$ if and only if $\s$ is an $\ell$-sequence, i.e.,
recursively defined through $s_0=0$, $s_1=1$, and $s_{i}=\ell s_{i-1}-s_{i-2}$ for $i \ge 2$.
One consequence is that among such sequences $\s$, unless $\s$ is an $\ell$-sequence, the generating function for the $\s$-lecture hall partitions can have
the form $((1-q^{e_1})(1-q^{e_2}) \cdots (1-q^{e_n}))^{-1}$ for at most finitely many $n$.

We also apply the results to   establish several conjectures by Pensyl and Savage regarding the symmetry of $h^*$-vectors for $\s$-lecture hall polytopes.
We end with open questions and directions for further research.
\end{abstract}

\maketitle

\tableofcontents


\section{Overview}
We will use polyhedral geometry to make progress on some open questions in the theory of {\em lecture hall partitions}.

\subsection{Lecture hall partitions and generating functions}
For a sequence $\s = \{s_i\}_{i \geq 1}$ of positive integers,
let $\LL_n^{(\s)}$ denote the set of all \emph{$\s$-lecture hall partitions} of length $n$,
\begin{equation}
\LL_n^{(\s)}=\left \{\lambda\in\integers^n:0 \leq \frac{\la_{1}}{{s_1}}\leq \frac{\la_{2}}{{s_2}} \leq \cdots\leq \frac{\la_{n}}{{s_n}}\right \} .
\label{lhpdef}
\end{equation}
When $\s$ is nondecreasing, each $\la \in \LL_n^{(\s)}$ satisfies
$\la_1 \leq \la_2 \leq \cdots \le \la_n$ and can be regarded as a {\em partition} of the integer $|\la| = \la_1 + \cdots + \la_n$.  But when $\s$ in not nondecreasing, the $\s$-lecture hall partitions are not necessarily {partitions}.
For example,
if $\s=(1,9,3,4)$ then $(0,3,1,2)$ and $(0,2,1,3)$ are distinct elements of $\LL_4^{(1,9,3,4)}$, even though they represent the same partition of the integer 6. 
In \cite{BME1} and \cite{BME2}, Bousquet--M\'elou and Eriksson consider the
generating function for $\s$-lecture hall partitions,
\[
f_n^{(\s)}(q) \  = \  \sum_{\la \in \LL_n^{(\s)}} q^{|\la|} ,
\]
where $|\la|=\la_1 + \cdots + \la_n$.
In \cite{BME1}, they show that for the sequence $\s=(1,2, \ldots, n)$ (this sequence gave rise to the name \emph{lecture hall partition} since one
interprets the parts as admissible heights of seats in a lecture hall), this generating function has the
form
\begin{equation}\label{LHT}
f_n^{(1,2, \ldots, n)}(q) \  =\sum_{\la \in \LL_n^{(1,2, \ldots, n)}} q^{|\la|} \  =  \ \prod_{i=1}^n \frac{1}{1-q^{2i-1}} \, .
\end{equation}

In partition theory, this result is notable not only because of the surprisingly simple generating function,
but also because it  is an entirely new finite form of {\em Euler's Partition Theorem}, which asserts
that the number of partitions of an integer $M$ into distinct parts is the same as the number of partitions of $M$
into odd parts. Note that the right-hand side of (\ref{LHT}) is the generating function for partitions (of any integer)
into parts from the set $\{1,3, \ldots, 2n-1\}$, which approaches the set of all odd parts as $n \rightarrow \infty$.
Correspondingly, on the left-hand side, as $n \rightarrow \infty$, the set $\LL_n^{(1,2, \ldots, n)}$ becomes the set of partitions into
distinct parts.

In \cite{BME2}, Bousquet--M\'elou and Eriksson show that a similar phenomenon occurs for a more general class of sequences which they refer to as $(k,\ell)$-sequences.
Given positive integers $k$ and $ \ell$, the \emph{$(k,\ell)$-sequence} $\a=(a_i)_{i=0}^\infty$ is defined by $a_0=0$, $a_1=1$, and for $i \geq 1$,
\[
a_{2i}=\ell a_{2i-1}-a_{2i-2} \qquad \text{ and } \qquad a_{2i+1}=k a_{2i}-a_{2i-1} \, .
\]
The following generalization of \eqref{LHT}  was proved in \cite{BME2}.
\begin{theorem}[Bousquet--M\'elou and Eriksson \cite{BME2}]
For positive integers $k, \ell \geq 2$, let $\a$ be the $(k,\ell)$-sequence and let $\b$ be the corresponding $(\ell,k)$-sequence.
Then
\begin{equation}\label{k_ell_gf}
f_n^{(\a)}(q) = \sum_{\la \in \LL_n^{(\a)}} q^{|\la|} =
\left \{
\begin{array}{ll}
\prod_{i=1}^{n}\frac{1}{1-q^{a_{i}+b_{i-1}}} & {\mbox {if $n$ is even,}}\\
\prod_{i=1}^{n}\frac{1}{1-q^{b_{i}+a_{i-1}}} & {\mbox {if $n$ is odd.} }\\
\end{array}
\right .
\end{equation}
\label{thm:k_ell_gf}
\end{theorem}
Note that if $k=\ell$, then $\a=\b$ and the generating function of the theorem becomes
\begin{equation}\label{ell_gf}
\prod_{i=1}^n \frac{1}{1-q^{a_i+a_{i-1}}} \, ,
\end{equation}
which simplifies to \eqref{LHT} when $k=\ell=2$.
When  $k=\ell$, we call $\a$ an \emph{$\ell$-sequence} \cite{SY}.  Thus the $\ell$-sequences are defined for a positive integer $\ell\geq 2$ by
\begin{equation}\label{l-sequencesdef}
  a_{i+1} = \ell a_{i} - a_{i-1} \, , 
\end{equation}
with initial conditions $a_1=1$, $a_0=0$.

To appreciate the significance of 
\eqref{LHT} and \eqref{k_ell_gf}, note that for a general sequence $\s$ of positive integers, the $\s$-lecture hall partitions have a generating function of the form
\begin{equation}
 f_n^{(\s)}(q) = \frac{H(q)}{\prod_{i=1}^n(1-q^{s_i+ \cdots + s_n})} \, ,
\label{horrible_gf}
\end{equation}
where $H(q)$ is a polynomial with nonnegative integer coefficients satisfying
$H(1)= \prod_{i=1}^n s_i$.
(See, e.g., \cite[Thm.~5]{SS}, where $H(q)$ is given a combinatorial interpretation.)
So, in the case of $(k,\ell)$-sequences, the numerator $H(q)$ factors and is completely cancelled by the denominator.

It is natural to consider if there might be other sequences $\s$ for which $f_n^{(\s)}$ would have the form
\begin{equation}
 f_n^{(\s)}(q) = \prod_{i=1}^n\frac{1}{(1-q^{e_i})}
\label{simple_gf}
\end{equation}
for some positive integers $e_1,\ldots,e_n$.
In \cite{BME2}, Bousquet--M\'elou and Eriksson investigate  sequences $\s$ having this property.
Their approach (and ours) is to study self-reciprocal generating functions.

\subsection{Lecture hall partitions and self-reciprocal generating functions}

A rational function $r(q)$ is {\em self-reciprocal} if there exists a nonnegative integer $k$ such that
\[
r(\tfrac 1 q) = \pm q^k \, r(q) \, .
\]
Note that if $f_n^{(\s)}(q)$ is of the form \eqref{simple_gf},  then 
$f_n^{(\s)}(q)$ is self-reciprocal, since
\[
f_n^{(\s)}\left(\frac{1}{q}\right) = \frac{(-1)^nq^{e_1+e_2+\cdots+e_n}}{\prod_{i=1}^n(1-q^{e_i})} \, .
\]
This led Bousquet--M\'elou and Eriksson to investigate in \cite{BME2} the relationship between the condition defining the $(k,\ell)$-sequences and the property that the generating function $f_n^{(\s)}(q)$ is self-reciprocal.
Bousquet--M\'elou and Eriksson define a sequence $\s$ to be {\em polynomic} if $f_n^{(\s)}(q)$ is the multiplicative inverse of a polynomial; hence, the $(k,\ell)$-sequences are polynomic.
They conjecture that in some sense all polynomic sequences (and, consequently, their self-reciprocal generating functions) arise from $(k,\ell)$-sequences, and they prove the following partial characterization.

\begin{theorem} [Bousquet--M\'elou and Eriksson \cite{BME2}]\label{thm:BMErec}
If $\s$ is a non-decreasing sequence of positive integers with the property that
$\gcd(s_i,s_{i-1})=1$ for $1<i\leq n$, then $f_n^{(\s)}(q)$ is self-reciprocal
if and only if 
$s_i+s_{i-2}$ is a multiple of $s_{i-1}$,
for $3 \leq i \leq n$, and
$s_2+1$ is a multiple of $s_1$. 
\label{BMEsr}
\end{theorem}

\begin{remark}
The original statement of Theorem~\ref{thm:BMErec} was a multivariate result.
The equivalence of this with the univariate statement above follows from Theorem~\ref{thm:StanGor} below.
\end{remark}

Theorem \ref{BMEsr} can be applied to show  that the $(k,\ell)$-sequences  have self-reciprocal generating functions.
It also implies that the generating function for the $(1,3,5,7)$-lecture hall partitions 
is self-reciprocal.  
However, the sequence $\s=(1,3,5,7)$ is not polynomic: it was shown in \cite{BME2} that
\[
f_4^{(1,3,5,7)}(q)
 = \frac{H(q)}{(1-q^7)(1-q^{12})(1-q^{15})(1-q^{16})} \, 
 = \frac{1-q+q^3-q^4+q^5-q^7+q^8}{(1-q)^2(1-q^{12})(1-q^{16})} \, ,
\]
where
\begin{align}
H(q) & = {q}^{28}+{q}^{27}+{q}^{26}+2\,{q}^{25}+2\,{q}^{24}+3\,{q}^{23}+4\,{q}^{22}+3\,{q}^{21}+4\,{q}^{20} \nonumber\\
& +5\,{q}^{19}+5\,{q}^{18}+6\,{q}^{17}+6\,{q }^{16}+6\,{q}^{15} +
 {q}^{14}+6\,{q}^{13}+6\,{q}^{12}+6\,{q}^{11} \label{Hq} \\
& +5\, {q}^{10}+5\,{q}^{9}+4\,{q}^{8}+3\,{q}^{7}+4\,{q}^{6}+3\,{q}^{5}+2\,{q}^{4}+2\,{q}^{3}+{q}^{2}+q+1 \, . \nonumber
\end{align}
Note that the coefficients of $H(q)$ are nonnegative.

Theorem \ref{BMEsr} can also be applied to get negative results.
For example,  for the first five terms of the Fibonacci sequence $\s=(1,1,2,3,5)$,
the generating function $f_5^{(\s)}(q)$ is not self-reciprocal
since $5+2$ is not a multiple of $3$.  As a consequence,
$f_5^{(\s)}(q)$ cannot have the form \eqref{simple_gf}.

But many sequences are not covered by Theorem \ref{BMEsr}.  For example,
$\s=(1,3,2,1,3,2)$ is not monotone;   $\s=(1,3,18,81,405,1944)$ does not satisfy
$\gcd(s_i,s_{i+1})=1$. We will show in Section 2 that both sequences give rise to $\s$-lecture hall partitions with self-reciprocal generating functions.
One of our main contributions is
the following result, which is implied by Theorem~\ref{Gorenstein_fae} (in Section
\ref{overviewgorcone}) and Theorem \ref{thm:lseqgorconj} (in Section \ref{gorensteinconditionsection}).
\begin{theorem}\label{thm:main}
Let $\ell > 0$ and $b \not = 0$ be integers satisfying $\ell^2+4b \geq 0$.
Let $\s$ be defined by the recurrence
\begin{equation}\label{recdef}
  s_n = \ell s_{n-1} + b s_{n-2} \, , 
\end{equation}
with initial conditions $s_1=1$, $s_0=0$.
The lecture hall generating function $f_n^{(\s)}(q)$ is self-reciprocal for all $n \geq 0$ if and only if $b=-1$.
If $b \not = -1$, there is an integer $n_0 = n_0(b,\ell)$ so that for all $n \geq n_0$,
 $f_n^{(\s)}(q)$ is not self-reciprocal and, consequently, $f_n^{(\s)}(q)$ cannot have the form
$((1-q^{e_1})(1-q^{e_2}) \cdots (1-q^{e_n}))^{-1}$.
\end{theorem}

The $\ell$-sequences have special significance in partition theory, giving rise to an $\ell$-version of Euler's partition theorem \cite{BME2, SY}.
But it has not been clear whether other sequences might behave similarly, or what other properties might distinguish $\ell$-sequences.
Theorem~\ref{thm:main} demonstrates that $\ell$-sequences play a characterizing role in the study of $\s$-lecture hall partitions.

The proof of Theorem~\ref{thm:main}, and of our extension of Theorem~\ref{BMEsr} to more general sequences,
involves properties of \emph{Gorenstein cones}.
The Gorenstein condition provides a framework for relating self-reciprocity of generating functions to arithmetic properties of the integer points in a rational pointed cone, a framework that has been established using the theory of normal semigroup algebras in combinatorial commutative algebra.

\subsection{Gorenstein lecture hall cones }\label{overviewgorcone}

In partition theory, the defining constraints for $\LL_n^{(\s)}$ are unusual, dealing with the {\em ratio} of consecutive parts of the partition $\la$, rather than  more typical constraints, such as the difference of consecutive parts, or the set of allowable parts.
However, within the theory of lattice-point enumeration in rational polyhedral cones, the lecture-hall constraints are just special cases of
general linear constraints bounding a region that contains the points of interest.
This is not a new observation; in \cite[Section 5]{BME2}, Bousquet--M\'elou and Eriksson point out that the study of lecture hall partitions falls naturally within the theory of linear homogeneous diophantine systems of inequalities.
However, this observation provides a wide variety of algebraic and combinatorial tools for applications, including the Gorenstein property.

Recall that a \emph{polyhedral cone} in $\reals^n$ is the solution set to a finite collection of linear inequalities $Ax\geq 0$ for some real matrix $A$.
The cone is \emph{rational} if $A$ has rational entries, it is \emph{simple} if it can be described by $n$ inequalities, and it is \emph{pointed} if it does not contain a linear subspace of $\reals^n$.
A pointed rational cone $\CC \subset \reals^n$ is \emph{Gorenstein} if there exists an integer point $\c$ in the interior~$\CC^\circ$ of $\CC$ such that $\CC^\circ \cap \integers^n = \c+(\CC \cap\integers^n)$.
We call $\c$ a \emph{Gorenstein point} of $\CC$.


For a sequence $\s = \{s_i\}_{i \geq 1}$ of positive integers,
the {\em $\s$-lecture hall cone}
$\CC_n^{(\s)}$ is defined by
\[
\CC_n^{(\s)}=\left\{\lambda\in\reals^n:0 \leq \frac{\la_{1}}{{s_1}}\leq \frac{\la_{2}}{{s_2}} \leq \cdots\leq \frac{\la_{n}}{{s_n}}\right\} \, .
\]
Observe that $\CC_n^{(\s)}$ is a pointed rational cone and
$ \LL_n^{(\s)} = \CC_n^{(\s)} \cap \integers^n$.
The Gorenstein condition for $\CC_n^{(\s)}$ appears implicitly in the work of Bousquet--M\'elou and Eriksson \cite{BME2}, where the Gorenstein point of $\CC_n^{(\s)}$ is referred to as the \emph{minimal strict lecture hall partition}.
The key feature of Gorenstein cones as they appear in our context is the following.
Set
\[
g^{(\s)}(x_1,\ldots,x_n)=\sum_{\la \in \LL_n^{(\s)}}x_1^{\la_1}x_2^{\la_2}\cdots x_n^{\la_n} \, .
\]

\begin{theorem} \label{thm:gorsym}
For a sequence $\s$ of positive integers, write as in \eqref{horrible_gf} 
\[
f_n^{(\s)}(q) = \frac{H(q)}{\prod_{i=1}^n(1-q^{s_i+ \cdots + s_n})} \, .
\]
The following are equivalent:
\begin{itemize}
\item  $\CC_n^{(\s)}$ is Gorenstein;
\item there exists $m\in \LL_n^{(\s)}$ such that $g^{(\s)}(\frac 1 {x_1},\ldots, \frac 1 {x_n})=\pm x_1^{m_1}\cdots x_n^{m_n}g^{(\s)}(x_1,\ldots,x_n)$;
\item $f_n^{(\s)}(q) =  \sum_{\la \in \LL_n^{(\s)}} q^{|\la|}$ is self-reciprocal.
\end{itemize}
\label{Gorenstein_fae}
\end{theorem}

The last condition is equivalent to $H(q)$ being a \emph{palindromic} polynomial, i.e., its coefficient sequence reads the same from left to right as from right to left.
Theorem~\ref{thm:gorsym} explains the behavior of the example following Theorem \ref{BMEsr}, where we observed that the  $(1,3,5,7)$-lecture hall partitions have a self-reciprocal generating function, and its corresponding $H(q)$ shown in \eqref{Hq} is palindromic.
A proof of Theorem~\ref{thm:gorsym} is given in Section~\ref{backgroundsection}.
In Section~\ref{testingsection}, we develop tools for checking the Gorenstein condition for $ \CC_n^{(\s)}$.
In Section~\ref{constructsection}, we use these tools to generalize Theorem \ref{BMEsr} and construct sequences that give rise to Gorenstein lecture hall cones.  

Theorem \ref{BMEsr} leaves open the possibility that sequences of the form~\eqref{recdef}, other than $\ell$-sequences,  could have
self-reciprocal generating functions.
For example, when $\ell=3$ and $b=9$, $\s=(1,3,18,81,405,1944, \ldots)$.
As will be shown in Section~\ref{sec:gorcones}, $f_6^{(\s)}(q)$ is self-reciprocal, but $f_7^{(\s)}(q)$ is not.
In Section~\ref{gorensteinconditionsection}, we prove Theorem~\ref{thm:lseqgorconj}, which states that the $\ell$-sequences are unique among second order linear recurrences in the following sense:
$\CC_n^{(\s)}$ is Gorenstein for all $n \geq 1$ if and only if $\s$ is an $\ell$-sequence.
The proof relies on the tools developed in Section~\ref{testingsection} together with a delicate analysis of how $\gcd(s_i,s_{i+1})$ is related to $b$ and $\ell$. 

Another consequence of our main results involves the $h^*$-vector of certain polytopes associated with the $\s$-lecture hall partitions.
For a sequence $\s = \{s_i\}_{i \geq 1}$ of positive integers,
the {\em rational lecture hall polytope} $\RR_n^{(\s)}$ is defined by
\[
\RR_n^{(\s)} = \left\{ \la \in \CC_n^{(\s)}  : \, \la_n \leq 1 \right\}.
\]
Note that $\RR_n^{(\s)}$ is a rational simplex (i.e., the convex hull of $n+1$ affinely independent points in $\reals^n$) because it is an $n$-dimensional object
defined by $n+1$ linear inequalities.
In Section~\ref{ehrhartsection}, we define the $h^*$-vector of  $\RR_n^{(\s)}$ and show it is symmetric if and only if $\CC_n^{(\s)}$ is Gorenstein.
This settles some conjectures from~\cite{PS}.


\section{Gorenstein $\s$-lecture hall cones}\label{sec:gorcones}


\subsection{Background}\label{backgroundsection}

All of our results depend on the following theorem.

\begin{theorem}[Stanley \cite{StanleyHilbFns}]\label{thm:StanGor}
Let $\CC  \subseteq \reals^n$  be an $n$-dimensional pointed rational cone and let $\ve e_1,\ldots,\ve e_n$ denote the standard basis for $\reals^n$.
Let $w$ be a linear functional from $\reals^n$ to $\reals$ such that for all $\ve u\in \CC \cap \integers^n$, we have
\begin{itemize}
\item $w({\ve u})$ is a nonnegative integer,
\item $w({\ve u)}= 0$ implies ${\ve u}=0$, and
\item for any nonnegative integer $a$, $w^{-1}(a)\cap \CC$ is a finite set.
\end{itemize}
Let 
\[
f(q) = \sum_{{\ve u} \in \CC \cap \integers^n} q^{w({\ve u})} \phantom{..} \text{ and }
 \phantom{..} g(x_1,\ldots,x_n) = \sum_{\ve u \in \CC \cap \integers^n} x_1^{u_1w(\ve e_1)}\cdots x_n^{u_nw(\ve e_n)} \, .
\]  
Then $\CC$ is Gorenstein if and only if, as rational functions, one (and hence both) of the following conditions hold: 
\begin{itemize}
\item $f(q)=(-1)^{n}q^b f(\tfrac 1 q)$ for some nonnegative integer $b$.
\item there exists ${ \ve m}\in \CC \cap \integers^n$ such that $g(\frac 1 {x_1},\ldots, \frac 1 {x_n})=\pm x_1^{m_1}\cdots x_n^{m_n}g(x_1,\ldots,x_n)$;
\end{itemize}
\end{theorem}

Theorem~\ref{thm:StanGor} as stated looks much different than how it appears originally \cite[Theorems 4.4 and 6.1]{StanleyHilbFns}, as statements about finitely generated graded algebras.
We will briefly explain how the commutative algebra comes into play here, via theorems due to Richard Stanley and Mel Hochster; we refer the reader to the comprehensive monograph \cite{CMRings} for a textbook account.
From $\CC$ as given in Theorem~\ref{thm:StanGor}, a semigroup algebra $\mathbb{C}[\CC]=\mathbb{C}[\CC\cap \integers^n]$ can be formed.
This algebra is isomorphic to the algebra of polynomials generated by monomials whose exponent vectors are contained in $\CC\cap \integers^n$.
The linear function $w$ induces an $\mathbb{N}^n$-multigrading on $\mathbb{C}[\CC]$, which by Gordan's Lemma \cite[Section 6.1]{CMRings} is finitely generated -- this multigrading may be coarsened to an $\mathbb{N}$-grading in at least one way.
A finitely generated $\mathbb{N}$-graded $\mathbb{C}$-algebra $A$ is called \emph{Cohen-Macaulay} if there exists a polynomial subalgebra $\mathbb{C}[\theta_1,\ldots,\theta_n]\subset A$, generated by a system of parameters for $A$ composed of homogeneous elements $\theta_1,\ldots,\theta_n\in A$, over which $A$ is a finitely generated free module.
The class of \emph{Gorenstein} algebras (algebraically defined) is a sub-class of Cohen-Macaulay algebras, but we omit the definition here because it is rather technical.

The following is a special case of a theorem due to Mel Hochster.
\begin{theorem}[Hochster \cite{Hochster}]
For every pointed rational cone $\CC$, there is at least one linear functional $w$ leading to an $\mathbb{N}$-grading for which $\mathbb{C}[\CC]$ is Cohen-Macaulay.
\end{theorem}
The next result, due to Richard Stanley \cite[Theorem 6.7]{StanleyHilbFns}, shows that the algebraic condition of $\mathbb{C}[\CC]$ being Gorenstein is equivalent to the geometric condition for Gorenstein introduced earlier.
\begin{theorem}[Stanley \cite{StanleyHilbFns}]
$\mathbb{C}[\CC]$ is Gorenstein (algebraically defined) if and only if $\CC$ is Gorenstein.
\end{theorem}
Theorem~\ref{thm:StanGor} follows by combining these two theorems with two more results of Stanley \cite[Theorems 4.4 and 6.1]{StanleyHilbFns}.
These characterize the (algebraic) Gorenstein condition for $\mathbb{C}[\CC]$ in terms of the self-reciprocity of the Hilbert function of $\mathbb{C}[\CC]$, in both the uni- and multi-variate case.
Since these Hilbert functions of $\mathbb{C}[\CC]$ are precisely the $f(q)$ and $g(x_1,\ldots,x_n)$ given above, Theorem~\ref{thm:StanGor} follows.

We now apply Theorem~\ref{thm:StanGor} to prove Theorem~\ref{thm:gorsym}.

\begin{proof}[Proof of Theorem~\ref{thm:gorsym}]
The weight function $w(\lambda)= |\lambda| = \sum_i \lambda_i$ satisfies the assumptions of Theorem~\ref{thm:StanGor}, and thus the first three conditions are equivalent.
The rationality of $f(q)$ and $g(x_1,\ldots,x_n)$ are well-known properties of integer point transforms for pointed rational cones \cite[Chapter 3]{ccd}.
\end{proof}


\subsection{Testing for the Gorenstein property}\label{testingsection}

The following lemma demonstrates how one can construct Gorenstein points for Gorenstein $\s$-lecture hall cones.

\begin{lemma}\label{lem:greedygorensteinlemma}
  Let $\CC = \{ \lambda\in\reals^n : \, A\lambda \geq 0 \}$ be a simple polyhedral
  cone, where $A$ is a lower triangular matrix with positive entries on the
  diagonal.  
  Denoting the rows of~$A$ as linear functionals
  $\alpha^1,\dots,\alpha^n$ on $\reals^n$, define a point $\ve c\in \CC^\circ$ by
  the following algorithm.
  For $1 \leq i \leq n$, choose $c_{i}\in\integers$ minimal so that $\alpha^i(\ve c)
  = \alpha^i(c_1,\dots, c_i) > 0$.  (This choice is possible because $A$ has
  positive entries on the diagonal.) 
  If $\CC$ is Gorenstein, then $\c$ is the unique Gorenstein point of~$\CC$.
\end{lemma}

\begin{proof}
  The point~$\c$ lies in the interior of~$\CC$ because $\alpha^i(\ve c) > 0$ for
  $i=1,\dots,n$.  Let $\CC$ be Gorenstein, and let $\hat\c$ be a Gorenstein
  point.  We prove the following property for all $j=0,\dots,n$ by induction:
  $\hat c_i = c_i$ for all $i\leq j$.  For $j=0$, nothing is to be shown.
  For $j>0$, since $\c \in \hat\c + (\CC\cap\integers^n)$ and  $\hat\c$ is in the interior of $\CC$, we find $$\alpha^j(c_1,\dots,c_{j-1},c_j)
  \geq \alpha^j(\hat c_1,\dots,\hat c_{j-1},\hat c_j) = \alpha^j(c_1,\dots,c_{j-1},\hat
  c_j) > 0.$$ 
  Thus $c_j \geq \hat c_j$, and, due to the minimal choice of $c_j$, actually
  $c_j = \hat c_j$. 
\end{proof}

The next lemma provides one of our main tools for checking the Gorenstein condition of an $\s$-lecture hall cone.


\begin{lemma}\label{lemma:shiny-gorenstein}
Let $\CC = \{ \lambda\in\reals^n : \, A\lambda \geq 0 \}$ be a full-dimensional simple polyhedral cone, 
where $A$ is a rational matrix and denote the rows of~$A$ as linear functionals $\alpha^1,\dots,\alpha^n$ on $\reals^n$.
For $j=1,\dots,n$, let the projected lattice  $\alpha^j (\integers^n) \subset \reals$ 
be generated by the number $q_j\in\rationals_{ >0 }$, so $\alpha^j (\integers^n) = q_j \integers$.  

\begin{enumerate}
\item
Then $\CC$ is Gorenstein if and only if there exists $\ve c \in \integers^n$
such that
$\alpha^j(\ve c) = q_j$ for all $j=1,\dots,n$.

\item
Define a point~$\ve{\tilde
    c} \in \CC\cap\rationals^n$ by $\alpha^j(\ve{\tilde c}) = q_j$ for
all $j=1,\dots,n$.
  Then $\CC$ is Gorenstein if and only if $\ve{\tilde c}\in\integers^n$.
\end{enumerate}
\end{lemma}

\begin{proof}
  Let $\ve b^1,\dots,\ve b^n\in\integers^n$ be the primitive generators of the
  rational cone~$\CC$, i.e., each $\ve b^j$ is an integer vector whose coordinates do not have a common factor.
  We may (and will) assume further that $\ve b^1,\dots,\ve b^n$ are
  labeled such that we have the biorthogonality relation
  $\alpha^i(\ve b^j) = 0$ for $i\neq j$ and $\alpha^i(\ve b^i) > 0$.
  Without loss of generality, by scaling $\alpha^i$, we can assume that
  $\alpha^i(\ve b^i) = 1$.

  In order to prove the first statement, assume $\CC$ is Gorenstein, with Gorenstein point $\ve c$.
  Suppose that there exists an index $j$ such that $\alpha^j(\ve c) \neq q_j$
  (and so $\alpha^j(\ve c) > q_j$.)
  Because $\alpha^j (\integers^n) = q_j \integers$, there exists a point $\ve
  z \in \integers^n$ with $\alpha^j(\ve z) = q_j$.  For $i \neq j$, let
  $\kappa_i = 1+\max\{ 0, \lceil -\alpha^i(\ve z) \rceil \}$.  Then $\ve{\hat z}
  = \ve z + \sum_{i\neq j} \kappa_i \ve b^i \in \CC^\circ \cap \integers^n$ and
  $\alpha^j(\ve{\hat z}) = \alpha^j(\ve z) = q_j$.
  Since $\alpha^j(\ve{\hat z}) = q_j < \alpha^j(\ve c)$, we find that $\ve
  {\hat z}$ does not lie in $\ve c + \CC$, and thus $\CC$ is not Gorenstein.

  Now suppose that for some $\ve c \in \integers^n$, $\alpha^j(\ve c) = q_j$ for all $j=1,\dots,n$.
  Then  $\ve c \in \CC^\circ$.
  Let $\ve z\in
  \CC^\circ\cap \integers^n$, so $\alpha^j(\ve z) \in q_j\integers_{>0}$ for all $j=1,\dots,n$.
  Then $\alpha^j(\ve z-\ve c) \geq 0$ for all $j=1,\dots,n$, and so $\ve z \in
  \ve c+\CC$.  Thus $\CC$ is Gorenstein.

  The second statement follows easily.
\end{proof}

\begin{corollary}\label{cor:csrec}
For a positive integer sequence $\s$, the $\s$-lecture hall cone $\CC_n^{(\s)}$ is Gorenstein if and only if
there exists $\ve c \in \integers^n$ satisfying
\[
c_js_{j-1} = c_{j-1}s_j + \gcd(s_j,s_{j-1})
\]
for $j>1$, with $c_1=1$.
\end{corollary}

Note that the identity in Corollary \ref{cor:csrec} defines the numbers $c_j$ recursively, once more confirming that the Gorenstein point $\ve c$ is unique.
In \cite{BME2}, Bousquet--M\'elou and Eriksson proved a version of Corollary~\ref{cor:csrec} for non-decreasing sequences of positive integers
$\s$, using their proof of the generalized lecture hall theorem.   
Note in Corollary~\ref{cor:csrec} that we do not need the condition that $\s$ is non-decreasing.

\begin{proof}
  We use the characterization of Lemma~\ref{lemma:shiny-gorenstein}.  We have
  \begin{align*}
    \alpha^1 &= (\tfrac1{s_1},0,\dots,0), \\
    \alpha^2 &= (-\tfrac1{s_1},\tfrac1{s_2},0,\dots,0) \\
    \alpha^3 &= (0,-\tfrac1{s_2},\tfrac1{s_3},0,\dots,0) \\
    &\vdots\\
    \alpha^n &= (0,\dots,0,-\tfrac1{s_{n-1}},\tfrac1{s_{n}}) \, .
  \end{align*}
  So, $q_1=1/s_1$ and $q_j=1/\lcm(s_{j-1},s_j)$ for $2 \leq j \le n$.
  By Lemma~\ref{lemma:shiny-gorenstein}, $\CC$ is Gorenstein if and only if there exists $\ve c \in \integers^n$ satisfying
  \begin{align*}
    \alpha^1(\ve c) &= \frac1{s_1} \\
    \alpha^2(\ve c) &= \frac1{\lcm(s_1,s_2)} \\
    \alpha^3(\ve c) &= \frac1{\lcm(s_2,s_3)} \\
    &\vdots\\
    \alpha^n(\ve c) &= \frac1{\lcm(s_{n-1},s_{n})} \, .
  \end{align*}
Since  $\alpha^1(\ve c)= c_1/s_1$ and $\alpha^j(\ve c)= -c_{j-1}/s_{j-1} + c_{j}/s_{j}$ for $2 \leq j \le n$, this is equivalent to
$c_1=1$ and, for $2 \leq j \le n$,
\[
c_js_{j-1}-s_jc_{j-1} = \frac{s_{j-1}s_j}{\lcm(s_j,s_{j-1})} = \gcd(s_j,s_{j-1}) \, .
\qedhere
\]
\end{proof}

Corollary~\ref{cor:csrec} gives a recurrence for $c_j$ and therefore implies the following.

\begin{corollary}\label{cor:onenenough}
If for some $n$ the sequence $(c_1, \ldots, c_n)$ satisfies the condition of Corollary \ref{cor:csrec}, then $(c_1, \ldots, c_j)$ also satisfies it for all $1 \leq j \leq n$.
Thus, if an $\s$-lecture hall cone fails to be Gorenstein in some fixed dimension, it fails for all higher dimensions as well.
\end{corollary}

\noindent
{\bf Example.}  Consider the sequence $\s=(1,3,18,81,405,1944,9477, \ldots)$ defined by the linear recurrence~\eqref{recdef} with $\ell=3$ and $b=9$.
The sequence of greatest common divisors $\gcd(s_i,s_{i-1})_{i \geq 1}$  is $(1,1,3,9,81,81,243, \ldots)$.  Applying Corollary~\ref{cor:csrec} to compute
the Gorenstein point $\ve c$ of $\CC_n^{(\s)}$ gives $\ve c=(1,4,25,113,566,2717,\frac{26491}{2}, \ldots)$.  Thus, by Corollaries~\ref{cor:csrec} and~\ref{cor:onenenough}, $\CC_n^{(\s)}$ is Gorenstein if and only if $n \leq 6$.


%

\subsection{A construction for Gorenstein lecture hall cones}\label{constructsection}

In this section we apply Corollary~\ref{cor:csrec} to 
show that there are many general sequences $\s$ that give rise to Gorenstein cones.
This will slightly generalize  Theorem \ref{BMEsr}.

Say that a sequence $\s$ of positive integers is \emph{$\u$-generated} by a sequence $\u$ of positive integers if
$s_2=u_1s_{1}-1$ and
 $s_{i+1}=u_is_{i}-s_{i-1}$ for $i > 1$.
For example, the $(k,\ell)$-sequences are $\u$-generated by
$\u=(\ell+1,k,\ell,k,\ell, \ldots)$. 

We will prove the following generalization of Theorem \ref{BMEsr}:

\begin{theorem}
Let $\s=(s_1, s_2, \ldots, s_n)$ be a sequence of positive integers such that $\gcd(s_i,s_{i+1})=1$
for $1 \leq i < n$.  Then $\CC_n^{(\s)}$ is Gorenstein if and only if
$\s$ is $\u$-generated by some sequence
$\u=(u_1, u_2, \ldots, u_{n-1})$ of positive integers.
When such a sequence exists, the Gorenstein point $\c$ for $\CC_n^{(\s)}$ is defined by
$c_1=1$, $c_2=u_1$, and for $2 \leq i < n$, $c_{i+1}=u_ic_i-c_{i-1}$.
\label{genBMEsr}
\end{theorem}

Consider the sequence $\s=(1,3,2,1,3,2,1,3,2,1, \ldots)$ which is not monotone and therefore not covered by Theorem \ref{BMEsr}.
It  is $\u$-generated by $\u=(4,1,2,5,1,2,5,1,2,5, \ldots )$,
and therefore by Theorem~\ref{genBMEsr}, $\CC_n^{(\s)}$ is Gorenstein and
its Gorenstein point
 is
$\c=(1,4,3,2,7,5,3,10,7, \ldots)$.

As another example,  
the ``$1 \bmod k$'' sequences
\[
\s \ = \ (1, \ k+1, \ 2k+1, \ \ldots, \ (n-1)k + 1)
\]
were studied in \cite{gopal}, and in \cite{PS} it was conjecured that the 
$h^*$-vector of $\RR_n^{(\s)}$ is always symmetric.
These $1 \bmod k$  sequences are $\u$-generated by
$\u=(k+2,2,2,2,\ldots)$ and therefore, by Theorem~\ref{genBMEsr}, give rise to Gorenstein cones.
In Section~\ref{ehrhartsection}, we will show that this implies symmetry of the $h^*$-vector.
The Gorenstein point is $(1,k+2,2k+3,3k+4, \ldots)$.

\begin{proof}[Proof of Theorem~\ref{genBMEsr}]
We adapt the method of \cite{BME2}.
By Corollary~\ref{cor:csrec}, $\CC_n^{(\s)}$ is Gorenstein if and only if
there exists $\ve c \in \integers^n$ satisfying
\[
c_js_{j-1} = c_{j-1}s_j + \gcd(s_j,s_{j-1})
\]
for $j>1$, with $c_1=1$.
Since  $\gcd(s_i,s_{i+1})=1$, we have that, for $1 < j < n$,
\[
c_js_{j-1}-c_{j-1}s_{j} = 1 = c_{j+1}s_{j} - c_js_{j+1}.
\] 
So, we can conclude that $c_j$ and $s_j$ are relatively prime and
\[
c_j(s_{j-1}+s_{j+1})=s_j(c_{j+1}+c_{j-1}).
\]
But then, since $\gcd(c_j,s_j)=1$, $s_{j-1}+s_{j+1}$ must be a multiple of $s_j$, i.e., for some positive integer $u_j$,
\[
s_{j+1}=u_js_j-s_{j-1}.
\]
Since $c_2s_1=s_2c_1+1=s_2+1$, setting $u_1=(s_2+1)/s_1=c_2$ ensures that $\s$ is $\u$-generated.

For the Gorenstein point,  we have $c_1=1$, $c_2=u_1$ and for $j>1$,  
\[
s_j(c_{j+1}+c_{j-1}) = c_j(s_{j-1}+s_{j+1}) = c_ju_js_j,
\]
so $c_{j+1}+c_{j-1}=c_ju_j$, as claimed.
\end{proof}

\section{The Gorenstein condition and second order linear
recurrences}\label{gorensteinconditionsection}



Assume that for nonzero integers  $\ell$ and $b$, the sequence $\s$ is defined recursively through
\begin{equation}\label{eq:sseqdef}
  s_0= 0, \ s_1=1, 
  \qquad \text{ and } \qquad
  s_j=\ell s_{j-1}+b s_{j-2} \ \text{ for } \ j\geq 2 \, .
\end{equation}
If $\ell \geq 2$ and $b=-1$, $\s$ is an $\ell$-sequence as noted in~\eqref{l-sequencesdef}.
In this section we will prove that among sequences of {\em positive} integers
$\{s_j\}_{j \geq 1}$ defined by second order linear recurrences with initial
conditions 0 and 1, only the $\ell$-sequences give rise to Gorenstein
lecture hall cones in every dimension $n$.

We first establish the conditions on $\ell$ and $b$ which guarantee that 
$s_j > 0$ for all $j \geq 1$.
The {\em characteristic equation} of the recurrence \eqref{eq:sseqdef}    is $x^2 - \ell x - b$ and its roots $u$ and $v$ are
\[
u = \frac{\ell + \sqrt{\ell^2 + 4b}}{2}; \ \ \ \ 
v = \frac{\ell - \sqrt{\ell^2 + 4b}}{2}.  
\]
These roots define $s_n$, when $u \not = v$, by
\[ s_n = \frac{u^n-v^n}{u-v},
\]
and when  $u=v$, by
\[
 s_n = nu^{n-1}.
\]
The {\em discriminant} of the characteristic polynomial is $D= \ell^2+4b$.
The following relations can be easily checked:
$$u-v = \sqrt{D} > 0;\ \ \ \   u+v=\ell; \ \ \ uv=-b.$$

\begin{proposition}\label{prop:positiveterms}
Let $\s$ be defined through \eqref{eq:sseqdef}  for nonzero integers $\ell$ and $b$.
Then
$s_j > 0$ for all $j \geq 1$ if and only if $\ell > 0$ and $D=\ell^2 + 4b \geq 0$.
\end{proposition}
\begin{proof}
If $\ell^2 + 4b < 0$, then, as is shown in \cite[Lemma~5]{HHH},
there exists $j \geq 1$, such that $s_j < 0$. (Actually, this is shown to be true for any initial conditions).
If $\ell \leq 0$, then $s_2=\ell s_1 + bs_0 = \ell \leq  0$.

So assume $D \geq 0$ and $\ell > 0$.  If $D=0$, then $u= \ell/2$  and 
$s_n = n(\ell/2)^{n-1}$ which is positive for $n \geq 1$ since $\ell$ is.
Otherwise, $D>0$ and, since $\ell > 0$, we have $u > 0$ and  $|v/u|<1$.  So
we have
\[
s_n = \frac{u^n}{\sqrt{D}} \left ( 1 - (v/u)^n \right ) > 0.
\]
\end{proof}

Note that if $\ell > 0$, but $b=0$ in \eqref{eq:sseqdef}, then
$s_n = \ell^{n-1}$ for $n \geq 1$.  In this case, 
$\CC_n^{(\s)}$ is Gorenstein for all $n \geq 0$, with Gorenstein point
$\c=(c_1, \ldots, c_n)$, where $c_i= 1 + \ell + \ell^2 + \cdots + \ell^{n-1}$.

Our main contribution is the following.

\begin{theorem} \label{thm:lseqgorconj}
Let $\ell > 0$ and $b \not = 0$ be integers satisfying $\ell^2+4b \geq 0$.
Let $\s=(s_1, s_2, \ldots)$ be defined by
$$s_n = \ell s_{n-1} + b s_{n-2} \, ,$$
with initial conditions $s_1=1$, $s_0=0$.
Then  $\CC_n^{(\s)}$ is Gorenstein for all $n \geq 0$ if and only if $b=-1$.
If $b \not = -1$, there exists $n_0 = n_0(b,\ell)$ such that 
$\CC_n^{(\s)}$
 fails to be Gorenstein
for all $n \geq n_0$.
\label{ell-Gorenstein}
\end{theorem}
\noindent

The proof of Theorem \ref{thm:lseqgorconj}  depends completely  on Corollaries \ref{cor:csrec} and 
\ref{cor:onenenough}. The key is to understand exactly how
$\ell$ and $b$ affect $\gcd(s_{n+1},s_n)$. This is the focus of Section 3.1.
Theorem \ref{thm:lseqgorconj} is proved in Section 3.2


\subsection{The effect of $\ell$ and $b$ on $\gcd(s_{n+1},s_n)$}


For integers $x$ and $y$ with $x \not = 0$,
we use the notation $x|y$ to mean that $y$ is divisible by $x$.
The greatest common divisor of two integers $x$ and $y$, 
\[\gcd(x,y) = \max \{z \in \integers \ | \  z | x \ {\rm and} \ z | y \}, \]
is defined as long as at least one of $x,y$ is nonzero. It is always a positive integer. 

\begin{definition} 
For nonzero integers 
$\ell$ and $b$, 
define
\begin{align*}
r = &  \gcd(\ell,b); \\
t = &  \gcd(\ell^2/r,b/r); \\
\sigma = & r/t; \\
\gamma = & \ell/r;\\
\beta =& b/(rt). 
\end{align*}
\label{case2defs}
\end{definition}

Note that $r$, $t$, and $\sigma$ are positive integers, whereas 
$\gamma$ and $\beta$ are integers whose signs agree with those of
$\ell$ and $b$, respectively.  The facts in the following two propositions are easily checked.

\begin{proposition} In Definition \ref{case2defs},
\begin{align*}
 \gcd(\ell,b)  \ =& \   \ r  \ = \sigma t;\\
\gcd(\ell^2,b) \ =& \   \ rt  \ = \sigma t^2;\\
\ell \ =& \ r \gamma \  =  \ \sigma t \gamma;\\
b \ =& \ rt \beta \  =  \ \sigma t^2 \beta;\\
\gcd(\gamma,\beta) \ =&   \  1;\\
\gcd(\gamma,t) \ =&   \  1;\\
\gcd(\sigma,\beta) \  =& \  1.
\end{align*}
\label{ell_b_facts}
\end{proposition}

\begin{proposition}
If $\gcd(\ell,b)=\gcd(\ell^2,b)=r$, then  $t=1$
and $r=\sigma$.
In this case,
\begin{align*}
\ell \ =& \ r \gamma;\\
b \ =& \ r \beta; \\
\gcd(\gamma,\beta) \  =&\  1;\\
\gcd(r,\beta) \ =&\ 1.
\end{align*}
\label{t_eq_1_note}
\end{proposition}

\begin{lemma}
For $\s$ defined by \eqref{eq:sseqdef}, assume
that $\ell$ and $b$ are nonzero integers.
Then for $n \geq 1$, $s_n$ is divisible by $t^{n-1}\sigma^{\lfloor n/2 \rfloor}$.
\label{sdivisor}
\end{lemma}

\begin{proof}
We use induction on $n$.
When $n=1$, $t^{1-1}\sigma^{\lfloor \frac 1 2 \rfloor}= 1$, which divides (any integer) $s_1$.  For $n=2$, $t^{2-1}\sigma^{\lfloor 2/2 \rfloor}= t \sigma=r$,
which divides $\ell s_1 + bs_0=s_2$ (for any integers $s_0,s_1)$.

Let $n \geq 3$ and assume the claim true for smaller values.
Then by the recurrence for $\s$,
\[
s_n \ = \ \ell s_{n-1} + b s_{n-2} \ = \  
t \sigma \gamma  s_{n-1} + t^2 \sigma \beta s_{n-2}.
\]
If $n=2k+1$, then
\begin{equation*}\label{divides1}
\frac{s_{2k+1}}{t^{2k} \sigma^k} \ = \ 
\frac{t \sigma \gamma s_{2k}}{t^{2k} \sigma^k} +
\frac{t^2 \sigma \beta s_{2k-1}}{t^{2k} \sigma^k} 
\ = \ 
 \sigma \gamma \, \frac{ s_{2k}}{t^{2k-1} \sigma^k} +
 \beta  \,\frac{ s_{2k-1}}{t^{2k-2} \sigma^{k-1}}.
\end{equation*}
By induction, the two fractions on the right are integers, so
$s_{2k+1}$ is divisible by $t^{2k} \sigma^k$.

If $n=2k$, then
\begin{equation*}\label{divides2}
\frac{s_{2k}}{t^{2k-1} \sigma^k} \ =  \
\frac{t \sigma \gamma s_{2k-1}}{t^{2k-1} \sigma^k} +
\frac{t^2 \sigma \beta s_{2k-2}}{t^{2k-1} \sigma^k} 
\ = \
 \gamma  \,\frac{ s_{2k-1}}{t^{2k-2} \sigma^{k-1}} +
 \beta  \,\frac{ s_{2k-2}}{t^{2k-3} \sigma^{k-1}}.
\end{equation*}
By induction, the two fractions on the right are integers, so
$s_{2k}$ is divisible by $t^{2k-1} \sigma^k$.
\end{proof}

From Lemma \ref{sdivisor}, 
for $n \geq 1$, as long as $s_{n+1}$ and $s_n$ are not both 0,
\[
\gcd(s_{n+1},s_n) \ \ \ {\rm is \  divisible \  by} \ \ \ t^{n-1}\sigma^{\lfloor n/2 \rfloor}.
\]
We now show that in the special case 
$t=1$, we have equality.

\begin{lemma}
For $\s$ defined by \eqref{eq:sseqdef}, assume
that $\ell>0$ and $b \not = 0$ are integers satisfying $\ell^2+4b \geq  0.$
If $\gcd(\ell,b)= \gcd(\ell^2,b)=r$, then  for
$n \geq 1$,
\[
\gcd(s_{n+1},s_{n})= r^{\lfloor n/2 \rfloor}.
\]
\label{gcd_lemma_1}
\end{lemma}
\begin{proof}
Since $\gcd(\ell,b)=r=\gcd(\ell^2,b)$, we have, from Proposition \ref{t_eq_1_note},
\[
t=1; \ \  \ \ \ell=r \gamma; \ \  \ \  b= r \beta;
\]
and
\[
\gcd(r,\beta)=1; \ \   \ \  \gcd(\beta,\gamma)=1.
\]
We use induction on $n$.  For the base case,
$$
\gcd(s_2,s_1) \  = \ 
\gcd(\ell,1) \  =  \ 1  \ = r^{\lfloor \frac 1 2 \rfloor}
$$
and
$$
\gcd(s_3,s_2) \  = \ 
\gcd(\ell^2+b, \ell) \  = \  \gcd(b, \ell) \ = \ r \  = \  r^{\lfloor 2/2 \rfloor}.$$
Let $n \geq 3$ and assume the lemma is true for smaller values.
Suppose 
$$\gcd(s_{n+1},s_{n})= p \, r^{\lfloor n/2 \rfloor}.$$
We show that $p=1$.

Assume first that $n=2k$, so that our hypothesis  is 
\[
\gcd(s_{2k+1},s_{2k}) = p \, r^k.
\]
  By the recursion for $\s$,
\[
s_{2k+1} \ = \ r \gamma s_{2k} + r \beta s_{2k-1}.
\]
Since $pr^k$ divides both $s_{2k+1}$ and $s_{2k}$,
it must also divide $r \beta s_{2k-1}$.
By induction, $\gcd(s_{2k},s_{2k-1})= r^{k-1}$, so
$\frac{ s_{2k} }{ r^{k-1} }$ and $\frac{ s_{2k-1} }{ r^{k-1} }$ are
relatively prime and therefore
$pr$ must divide $r \beta$, i.e., $p$ divides $\beta $.

But now, applying the recursion for $\s$ again, we obtain
\[
s_{2k} \ = \ r \gamma s_{2k-1} + r \beta s_{2k-2}.
\]
By Lemma \ref{sdivisor},
$s_{2k-2}$ is divisible by $\sigma^{k-1}= r^{k-1}$.
Thus, since $p$ divides $\beta$, we have that $pr^k$ divides $r \beta s_{2k-2}$. 
So $pr^k$ also divides $s_{2k}$, it must
divide $r \gamma s_{2k-1}$.
By induction,  $\gcd(s_{2k-1},s_{2k})= r^{k-1}$, and thus
$pr$ is relatively prime to $\frac{ s_{2k-1} }{ r^{k-1} }$.
So it must be that $pr$ divides  $r \gamma$, i.e., $p$ divides $\gamma$.
But now we have that $p$ divides both $\beta$ and $\gamma$, which are relatively
prime.  So, $p=1$.

In the case $n=2k+1$, our hypothesis is
\[
\gcd(s_{2k+2},s_{2k+1}) = p r^k.
\]
  By the recursion for $\s$, we have
\[
s_{2k+2} \ = \ r \gamma s_{2k+1} + r \beta s_{2k}.
\]
Since $pr^k$ divides both $s_{2k+2}$ and $s_{2k+1}$,
it must also divide $r \beta s_{2k}$.
By induction, $\gcd(s_{2k+1},s_{2k})= r^{k}$, so
$\frac{ s_{2k+1} }{ r^{k} }$ and $\frac{ s_{2k} }{ r^{k} }$ are
relatively prime and therefore
$p$ must divide $r \beta$. 
Let $p'$ be a prime factor of $p$.

Suppose first that $p'$ does not divide $r$.
Then $p'$ divides $\beta$. 
Applying the recursion for $\s$ again,
\begin{equation}
s_{2k+1} \ = \ r \gamma s_{2k} + r \beta s_{2k-1}.
\label{apply_rec_again}
\end{equation}
By Lemma \ref{sdivisor},
$s_{2k-1}$ is divisible by $\sigma^{k-1}= r^{k-1}$.
Thus, since $p'$ divides $\beta$, we have that $p'r^k$ divides $r \beta s_{2k-1}$.
Since $p'r^k$ also divide $s_{2k+1}$, it must 
 divide $r \gamma s_{2k}$.
By induction,  $\gcd(s_{2k+1},s_{2k})= r^{k}$, and thus
$p'$ is relatively prime to $\frac{ s_{2k} }{ r^{k} }$.
So it must be that $p'$ divides  $r \gamma$. But our assumption was that the prime
$p'$ does not divide $r$, so it must be that $p'$ divides $\gamma$.
But then we have $p'$ dividing both $\beta$ and $\gamma$, which are relatively
prime.  So, $p'=1$.

On the other hand, given that $p'$ divides $r \beta$,
if  $p'$ divides $r$, then,
since $s_{2k}$ is divisible by $r^k$,
$r \gamma s_{2k}$ is divisible by $p'r^k$.
Since $p'r^k$ also divides $s_{2k+1}$, 
by (\ref{apply_rec_again}), it also divides $r \beta s_{2k-1}$.
But since $\gcd(s_{2k},s_{2k-1}) = r^{k-1}$,
$p'r$ is relatively prime to
$s_{2k-1}/r^{k-1}$  and therefore $p'r$ divides  $r \beta$,
so $p'$ divides  $\beta$.  But $\gcd(r,\beta)=1$, so $p'=1$. Thus $p=1$.
\end{proof}

By Lemma \ref{sdivisor}, for all integers $\ell$ and $b$ in our range of interest ($\ell > 0$ and $\ell^2+4b \geq 0$),
 $ \gcd(s_{n+1},s_n)$ is divisible by $t^{n-1}\sigma^{\lfloor n/2 \rfloor}$.
However, in contrast to the case when $t=1$,  equality does not always hold.  
For example, if $\ell=6$ and $b=36$, then $r=6$,  $t=6$, and $\sigma=1$,
so $t^{n-1}\sigma^{\lfloor n/2 \rfloor}=6^{n-1}$ and we get the following:

\vspace{.2in}
{\tiny
\noindent
\begin{tabular}{|c||c|c|c|c|c|c|c|c|c|c|c|c|c|c|c|c|c|c|c|c|c|c|c|c|}
\hline
\ & \ & \ & \ & \ & \ & \ & \ & \ & \ & \ & \ & \ & \ & \ & \ & \ & \ & \ & \ & \ & \ & \ & \ &\ \\
$n$ & 1 & 2 & 3&4&5&6&7&8&9&10&11&12&13&14&15&16&17&18&19&20&21&22&23&24\\
\ & \ & \ & \ & \ & \ & \ & \ & \ & \ & \ & \ & \ & \ & \ & \ & \ & \ & \ & \ & \ & \ & \ & \  &\ \\
\hline
\ & \ & \ & \ & \ & \ & \ & \ & \ & \ & \ & \ & \ & \ & \ & \ & \ & \ & \ & \ & \ & \ & \ & \  &\ \\
$\frac{\gcd(s_{n+1},s_n)}{6^{n-1}}$&
1& 1& 2& 3& 1& 2& 1& 3& 2& 1& 1& 6& 1& 1& 2& 3& 1& 2& 1& 3& 2& 1& 1& 6\\
\ & \ & \ & \ & \ & \ & \ & \ & \ & \ & \ & \ & \ & \ & \ & \ & \ & \ & \ & \ & \ & \ & \ & \  &\ \\
\hline
\end{tabular}
}

\vspace{.2in}
As another  example, if $\ell=6 \cdot 3 \cdot 5$ and $b=-36 \cdot 3 \cdot 7$, then $r=6 \cdot 3$,  $t=6$, and $\sigma=3$,
so $t^{n-1}\sigma^{\lfloor n/2 \rfloor}=6^{n-1}3^{\lfloor n/2 \rfloor}$ and we get the following:

\vspace{.2in}
{\tiny
\noindent
\begin{tabular}{|c||c|c|c|c|c|c|c|c|c|c|c|c|c|c|c|c|c|c|c|c|c|c|c|c|}
\hline
\ & \ & \ & \ & \ & \ & \ & \ & \ & \ & \ & \ & \ & \ & \ & \ & \ & \ & \ & \ & \ & \ & \ & \  &\ \\
$n$ & 1 & 2 & 3&4&5&6&7&8&9&10&11&12&13&14&15&16&17&18&19&20&21&22&23&24\\
\ & \ & \ & \ & \ & \ & \ & \ & \ & \ & \ & \ & \ & \ & \ & \ & \ & \ & \ & \ & \ & \ & \ & \  &\ \\
\hline
\ & \ & \ & \ & \ & \ & \ & \ & \ & \ & \ & \ & \ & \ & \ & \ & \ & \ & \ & \ & \ & \ & \ & \  &\ \\
$\frac{\gcd(s_{n+1},s_n)}{6^{n-1}3^{\lfloor n/2 \rfloor}}$&
1&1&2&1&1&6&1&1&2&1&1&6&1&1&2&1&1&6&1&1&2&1&1&6\\
\ & \ & \ & \ & \ & \ & \ & \ & \ & \ & \ & \ & \ & \ & \ & \ & \ & \ & \ & \ & \ & \ & \ & \  &\ \\
\hline
\end{tabular}
}
\vspace{.2in}

What we {\em will}
be able to show is that
$\frac{ \gcd(s_{n+1},s_n) }{ t^{n-1}\sigma^{\lfloor n/2 \rfloor} }$
is a
factor of $t$, something that might be conjectured from the evidence in the
tables.
In order to prove this, we need another fact.
\begin{lemma}
For $\s$ defined by \eqref{eq:sseqdef}, assume that $\ell > 0$ and $b \not = 0$ are integers satisfying
$\ell^2+4b\geq 0$.
Then for $n \geq 0$,
\[
s_n = t^{n-1} f_n,
\]
where $f$ is defined by the recurrence
\[
f_n \ = \  \tfrac{\ell}{t} \, f_{n-1} + \tfrac{b}{t^2} \, f_{n-2}, 
\]
with initial conditions $f_0=0$, $f_1=1$.  
Furthermore,
$f$, so defined, is an integer sequence satisfying
\[\gcd(f_{n+1},f_n) \ = \ \sigma^{\lfloor n/2 \rfloor}.
\]
\label{frec}
\end{lemma}
\begin{proof}
First, to show that $s_n=t^{n-1}f_n$, we use induction.
 For the base case, $s_0=0$, $s_1=1$, so the claim is true.
Let $n \geq 1$ and assume it is true for values less than or equal to $n$.
Then
\begin{align*}
s_{n+1}
 \ = \ 
\ell s_{n} + b s_{n-1}
 \ = \ 
\ell t^{n-1} f_{n} + b t^{n-2} f_{n-1}
 \ = \ 
t^{n}\left ( \tfrac{\ell}{t} \, f_n + \tfrac{b}{t^2} \,  f_{n-1} \right )
 \ = \ 
t^{n} f_{n+1}.
\end{align*}
To see that $f$ is an integer sequence,  from
Proposition~\ref{ell_b_facts},  $\frac{ \ell }{ t } = \sigma \gamma$ and $\frac{ b }{ t^2 } = \sigma \beta$,
and $\sigma$, $\gamma$, and $\beta$ are   integers.  In particular, $\ell/t$ is a positive integer.
Finally, to prove the assertion about $\gcd(f_{n+1},f_n)$,
observe that
\[
\gcd \left( \tfrac \ell t , \tfrac{ b }{ t^2 } \right) \ = \ \gcd(\sigma \gamma, \sigma \beta) \ = \ \sigma,
\]
since $ \gamma$ and $\beta$ are relatively prime.
Furthermore,
\[
\gcd \left( \left( \tfrac \ell t \right)^2 , \tfrac{ b }{ t^2 } \right) \ = \ \gcd \left( \sigma^2 \gamma^2, \sigma \beta \right) \ = \ \sigma,
\]
since $\beta$ is  relatively prime to both $\gamma$ and $\sigma$. 

Note that the discriminant of the characteristic polynomial of the recurrence for  $f$ is $(\ell/t)^2 + 4b/t^2 = (\ell^2+4b)/t^2$, which is nonnegative since  $\ell^2+4b \geq 0$.
Additionally, $\gcd \left( \frac \ell t , \frac{ b }{ t^2 } \right) = \gcd \left( \left( \frac \ell t \right)^2 , \frac{ b }{ t^2 } \right)$, so
 the recurrence $f$ satisfies the hypothesis of Lemma \ref{gcd_lemma_1}
and therefore the final claim follows.
\end{proof}

We can now show that
\[
t^{n-1}\sigma^{\lfloor n/2 \rfloor} \ \ \Big |  \ \ 
 \gcd(s_{n+1},s_n)  \ \ \Big | \ \  
t^{n}\sigma^{\lfloor n/2 \rfloor}.
\]

\begin{lemma}
For $\s$ defined by \eqref{eq:sseqdef}, assume
that $\ell> 0$ and $b \not = 0$ are integers satisfying  $\ell^2+4b \geq 0$.
For $n \geq 1$, 
$ \gcd(s_{n+1},s_n)$ divides $t^{n}\sigma^{\lfloor n/2 \rfloor}$.
\label{gcd_lemma_2}
\end{lemma}
\begin{proof}
Using Lemma \ref{frec},
\begin{align*}
\gcd(s_{n+1},s_n) \ &= \ \gcd(t^n f_{n+1},t^{n-1} f_n)
 \ = \ 
t^{n-1} \gcd(t f_{n+1},  f_n) \\ 
 \ &= \ t^{n-1} \sigma^{\lfloor n/2 \rfloor} \gcd \left (t \, \frac{f_{n+1}}{\sigma^{\lfloor n/2 \rfloor}},\frac{f_{n}}{\sigma^{\lfloor n/2 \rfloor}} \right ).
\end{align*}
By Lemma \ref{frec},  the fractions in the last line are relatively prime integers.
Thus,
\[
\gcd(s_{n+1},s_n) \ = \   
 t^{n-1} \sigma^{\lfloor n/2 \rfloor} \gcd \left (t,\frac{f_{n}}{\sigma^{\lfloor n/2 \rfloor}} \right ),
\]
which is a divisor of 
$ t^{n-1} \sigma^{\lfloor n/2 \rfloor}t= t^{n} \sigma^{\lfloor n/2 \rfloor}$.
\end{proof}

We need one more lemma to prove Theorem \ref{thm:lseqgorconj}, implying that the sequence
$\frac{ s_n }{ \gcd(s_n,s_{n+1}) }$ 
  grows without bound.

\begin{lemma}
For $\s$ defined by \eqref{eq:sseqdef}, assume
that $\ell> 0$ and $b \not = 0$ are integers satisfying  $\ell^2+4b \geq 0$.
Define the sequence $\{h_n\}_{n \geq 1}$ by
\[
h_n = \frac{s_n}{t^{n-1}\sigma^{\lfloor n/2 \rfloor}}.
\]
Then for any $B > 0$, there is a positive integer $n_0$ such that $h_n > B$ for all $n \geq n_0$.

\label{lemma:growth}
\end{lemma}
\begin{proof}
Assume first that $D=\ell^2+4b >0$. Then
\[
h_n = \frac{s_n}{t^{n-1}\sigma^{\lfloor n/2 \rfloor}} \geq
 \frac{s_n}{(t\sqrt{\sigma})^{n}} = \frac{1}{\sqrt{D}}\left (
\left ( \frac{u}{t \sqrt{\sigma}}\right )^n -
\left ( \frac{v}{t \sqrt{\sigma}}\right )^n 
\right ),
\]
where $u$ and $v$ are the roots of the characteristic equation and $u > 1$ and $|v|<|u|$.
Thus, we need only verify that $u/(t\sqrt{\sigma}) > 1$.  If $b > 0$, then, recalling Definition \ref{case2defs},
\[
\frac{u}{t \sqrt{\sigma}}\  =\  \frac{\ell + \sqrt{\ell^2+4b}}{2t \sqrt{\sigma}} \ >\ 
\frac{\ell}{t \sqrt{\sigma}}\  =\  \frac{\sigma t \gamma}{t \sqrt{\sigma}}\  = \ \gamma \sqrt{\sigma} \ \geq \  1,
\]
since both $\gamma$ and $\sigma$ are positive integers.
On the other hand, if $b <0$, then, from Definition \ref{case2defs},
$\ell = \sigma t \gamma$ and $b= \sigma t^2 \beta$, so
\[
\frac{u}{t \sqrt{\sigma}}\  =\ \frac{\sigma t \gamma + \sqrt{\sigma^2 t^2 \gamma^2 + 4 \sigma t^2 \beta}}{2 t \sqrt{\sigma}}
\ = \ \frac{\gamma \sqrt{\sigma} + \sqrt{\gamma^2 \sigma + 4 \beta}}{2} \ > \
\frac{\gamma \sqrt{\sigma}}{2}.
\]
But note that since $\ell^2 + 4b > 0$ and $b < 0$, it must be that
$\gamma^2 \sigma + 4 \beta > 0$ and $\beta < 0$.
This means that $\gamma^2 \sigma  > 4$ and therefore
$\gamma \sqrt{\sigma}/2 > 1$. 

For the case when $D=0$, note that $\ell$ must be even and $b=-\ell^2/4$ and
\[
s_n = n \left ( \frac{\ell}{2}\right ) ^{n-1}.
\]  
So, if $q=\ell/2$ is odd, then $r=\gcd(\ell,b) = \gcd(2q, -q^2) = q$
and $t=\gcd(\ell^2/r,b/r) = \gcd(4q, -q) = q.$
Thus $\sigma = 1$ and
\[ h_n = 
\frac{s_n}{t^{n-1} \sigma^{\lfloor n/2 \rfloor}} \ = \
\frac{n q^{n-1}}{q^{n-1}} \ = \ n.
\]
But if $D=0$ and $\ell/2=q$ is even,  then
$r=\gcd(\ell,b)=\gcd(2q, -q^2) = 2q$ and $t=\gcd(\ell^2/r,b/r)=\gcd(2q,q/2) = q/2$.
So $\sigma = r/t = 4$.  Then
\[ h_n = 
\frac{s_n}{t^{n-1} \sigma^{\lfloor n/2 \rfloor}} \ = \
\frac{nq^{n-1}}{(q/2)^{n-1} 4^{\lfloor n/2 \rfloor}} \ = \
\left \{
\begin{array}{ll}
n/2 & {\mbox {if $n$ is even,}}\\
n & {\mbox {if $n$ is odd.}}
\end{array}
\right .
\]
In either case, $h_n$ grows without bound.
\end{proof}

\subsection{Failure of the Gorenstein condition for $b \not =-1$}

We can now  complete the proof of Theorem \ref{thm:lseqgorconj}. 

\begin{theorem}
\label{thm:case2}
For integers $\ell > 0$ and $b \not = 0$ satisfying $\ell^2+4b\geq 0$, let
$\s$ be defined by
\begin{equation*}
  s_0= 0, \ s_1=1, 
  \qquad \text{ and } \qquad
  s_j=\ell s_{j-1}+b s_{j-2} \ \text{ for } \ j\geq 2 \, .
\end{equation*}
Choose $n_0$ large enough so that for all $n \geq n_0$,
\[
\frac{s_n}{t^{n-2}\sigma^{\lfloor (n-1)/2 \rfloor}} > t(r+|b|). 
\]
Then unless $b=-1$,
the $n$-dimensional $\s$-lecture hall cone $\CC_n^{ (\s) }$ is not Gorenstein 
for $n > n_0$. 
\end{theorem}
\begin{proof}
Note that by Lemma \ref{lemma:growth}, such choice of $n_0$ is always possible since
\[
\frac{ s_{n} }{t^{n-2} \sigma^{\lfloor (n-1)/2 \rfloor }} \  \geq \ 
\frac{ s_{n} }{t^{n-1} \sigma^{\lfloor n/2 \rfloor }} \  = \ h_n. 
\]

By Lemma \ref{gcd_lemma_2}, for each $n \geq 2$ there is a positive integer $u_n$ 
such that $u_n | t$ 
and
\begin{align}
\gcd(s_{n},s_{n-1}) \ = & \ u_n  t^{n-2}\sigma^{\lfloor n/2 \rfloor}; \label{ugcd}
\end{align}
Let $g_n = \gcd(s_{n},s_{n-1})$.

If $\CC^{(\s)}_{n}$ is Gorenstein,
then by Corollary \ref{cor:csrec}
 the Gorenstein point $\c$ satisfies
\begin{align}
c_{n}s_{n-1} \  =  \ & c_{n-1}s_{n} + g_n  \label{firstc}
\end{align}
and all coodinates of $\c$ are integers.
Rewriting (\ref{firstc}) gives
\[
c_{n} \, \frac{s_{n-1}}{g_n} \  - \ 
c_{n-1} \,\frac{s_{n}}{g_n} \  = \  1 \, .
\]
Note that by definition of $g_n$,  the fractions are integers
and therefore the integers $c_{n}$ and $\frac{s_{n}}{g_n}$
are relatively prime.
Under the assumption that $\CC^{(\s)}_{n+1}$ is also Gorenstein, we can combine
(\ref{firstc}) for consecutive $g_n, g_{n+1}$ to get
\[
g_{n+1}(c_{n}s_{n-1} - c_{n-1}s_{n}) \  = \ g_{n+1}g_n \  = \  g_n(c_{n+1}s_{n} -c_{n}s_{n+1}),
\]
so
\[
c_{n}(g_{n+1}s_{n-1}+g_ns_{n+1}) \  = \  s_{n}(g_nc_{n+1}+g_{n+1}c_{n-1}).
\]
Then, dividing through by $g_n$,
\[
c_{n} \left (g_{n+1} \frac{s_{n-1}}{g_n} + s_{n+1} \right ) \  = \  
\frac{s_{n}}{g_n}  \left ( g_n c_{n+1}+g_{n+1}c_{n-1}
\right ).
\]
Again, by definition of $g_n$,  all the fractions here are integers.
We now use the fact that 
$c_{n}$ and $\frac{s_{n}}{g_n}$
are relatively prime to conclude that
\[
\frac{s_{n}}{g_n} \ \ \  {\rm divides} \  \ \ 
\frac{g_{n+1} s_{n-1}}{g_n} + s_{n+1} \ =  \ 
\frac{g_{n+1} s_{n-1}}{g_n} + \ell s_{n} + bs_{n-1} \ = \  
s_{n-1} \left ( \frac{g_{n+1}}{g_n} + b \right ) + \ell s_n,
\]
where in the last steps we have applied the recursion for $\s$.
But now this means that 
\begin{equation*}
\frac{s_{n}}{g_n}  \ \ \  {\rm divides} \  \ \ s_{n-1} \left ( \frac{g_{n+1}}{g_n} + b \right ) .
\end{equation*}
Since 
\[
\gcd(s_{n},s_{n-1}) \ = g_n,
\]
the integers $\frac{ s_{n} }{g_n} $ and $s_{n-1}$ have no common factors and therefore
\begin{equation}
\frac{ s_{n} }{g_n} \ \ \ {\rm  divides} \ \ \ \  \frac{g_{n+1}}{g_n} + b .
\label{divides_condition}
\end{equation}
Recalling that $g_n = u_n t^{n-2} \sigma^{\lfloor (n-1)/2 \rfloor }$, we have that
\[
\frac{ s_{n} }{t^{n-2} \sigma^{\lfloor (n-1)/2 \rfloor }} \ \ \ {\rm  divides} \ \ \ \  \frac{u_{n+1}t^{n-1} \sigma^{\lfloor n/2 \rfloor }}{t^{n-2} \sigma^{\lfloor (n-1)/2 \rfloor }} + u_nb \ = \ 
u_{n+1}t \sigma^{\lfloor n/2 \rfloor - \lfloor (n-1)/2 \rfloor} + u_n b.
\]
But notice that for fixed $b$ and $\ell$, the right-hand side is bounded:
each $u_i$ is a positive divisor of $t$, so
\[
u_{n+1}t \sigma^{\lfloor n/2 \rfloor - \lfloor (n-1)/2 \rfloor} + u_n b
\leq t^2 \sigma + t |b| = tr+t|b| = t(r+|b|).
\]
On the other hand, the left-hand side satisfies 
\[
\frac{ s_{n} }{t^{n-2} \sigma^{\lfloor (n-1)/2 \rfloor }} \  >
t(r+|b|).
\]
Thus, for $n \geq n_0$, the only way that
\[\frac{ s_{n} }{t^{n-2} \sigma^{\lfloor (n-1)/2 \rfloor }} \ \ \ {\rm  divides}
\ \ \ \ 
u_{n+1}t \sigma^{\lfloor n/2 \rfloor - \lfloor (n-1)/2 \rfloor} + u_n b
\]
is if the latter is 0.

In other words, for $n \geq  n_0$, if 
$\CC^{(\s)}_{n+1}$ is Gorenstein, then
\begin{equation}
b = -\frac{u_{n+1}t \sigma^{\lfloor n/2 \rfloor - \lfloor (n-1)/2 \rfloor}}{u_n}.
\label{bcondition}
\end{equation}
Let $n \geq n_0$ be an odd integer  and assume 
$\CC^{(\s)}_{n+2}$ is Gorenstein. By Corollary \ref{cor:onenenough},
$\CC^{(\s)}_{n+1}$ is also Gorenstein. 
Then $b$ must satisfy \eqref{bcondition} for both $n$ and $n+1$, i.e.,
\begin{equation}
b = \frac{-u_{n+1}t}{u_n}
\label{firstb}
\end{equation}
and
\begin{equation}
b = \frac{-u_{n+2}t \sigma}{u_{n+1}}.
\label{secondb}
\end{equation}
Recall that every $u_i$ is a positive divisor of $t$ and that $b= rt\beta$.
So, from \eqref{firstb}, $-u_{n+1}/u_n = r \beta$  and therefore $u_n$ divides $u_{n+1}$.  But also, $r= t \sigma$, so $t \  |  \ (u_{n+1}/u_n)$.  But since 
$(u_{n+1}/ u_n) \  | \    t$,   we must have
$u_{n+1}/u_n = t$ and therefore $u_{n+1}=t$ and $u_n=1$.  This means that $b=-t^2$.

But \eqref{secondb}  must also hold, so
$b= -u_{n+2}t \sigma/u_{n+1}=-u_{n+2} \sigma$ since $u_{n+1}=t$.
On the other hand, since $b= \beta r t$, and $\sigma = r/t$, this means that
$\beta r t = -u_{n+1} r/t$ and so $\beta t^2 = -u_{n+1}$.
Thus, $t^2 | u_{n+1} | t$.  But $t$ is a positive integer, so this means $t=1$.
Combining with the conclusion $b=-t^2$ from the previous paragraph, we finally have that $b=-1$.
\end{proof}


\begin{corollary}
Let $b$, $\ell$, and $\s$ satisfy the hypothesis of Theorem \ref{thm:case2}.
If $\gcd(\ell,b)=\gcd(\ell^2,b)$ then 
when $b > 0$, 
$\CC^{(\s)}_n$ is not Gorenstein for $n \geq 5$ 
and when $b < -1$,
$\CC^{(\s)}_n$ is not Gorenstein for $n \geq 6$. 
\end{corollary}
\begin{proof}
If $\gcd(\ell,b)=\gcd(\ell^2,b)$ then, by Lemma \ref{gcd_lemma_1}, for $n \geq 1$,
$g_n = \gcd(s_n,s_{n-1}) = r^{\lfloor (n-1)/2 \rfloor}$.

First assume $b > 0$.  If $\CC_5^{(\s)}$ is Gorenstein, then by \eqref{divides_condition} in the proof of Theorem \ref{thm:case2},
\[
s_4 \ \Big | \ g_5 + bg_4 \  = \  r^2 + br \  = \  r(r+b).
\]
Since $b > 0$, $r+b \not = 0$, so $s_4 \leq r(r+b)$.
However, 
$$r(r+b) \  \leq \  \ell(\ell + b)  \ = \  \ell^2 + \ell b \  < \  \ell^3 + 2b\ell \  = \  s_4.$$
Thus, $\CC_5^{(\s)}$ is not Gorenstein and therefore, by Corollary \ref{cor:onenenough}, $\CC_n^{(\s)}$ is not Gorenstein for any $n \geq 5$.

Now assume $b < -1$.  
By \eqref{divides_condition}, 
If $\CC_6^{(\s)}$ is Gorenstein, then 
\[
s_5 \ \Big | \ g_6 + bg_5 \  = \  r^2 + br^2 \  = \  r^2(1+b).
\]
Since $b \not = -1$, this implies $s_5 \leq |r^2(1+b)|$.
However, $s_5 = \ell^4 + 3b\ell^2 + b^2$ and recall that we assume $\ell^2+4b \geq 0$.
Then, since $b < 0$, we have
\[
s_5 \ = \ \ell^2(\ell^2 + 4b) - b\ell^2 + b^2 \ \geq \ -b\ell^2 = \  |b|\ell^2.
\]
On the other hand,
\[
|r^2(1+b)| \ = \ r^2(|b|-1) \ < \ r^2|b| \  \leq \  \ell^2|b|,
\]
a contradiction.
Thus, $\CC_6^{(\s)}$ is not Gorenstein and therefore, by Corollary \ref{cor:onenenough}, $\CC_n^{(\s)}$ is not Gorenstein for any $n \geq 6$.
\end{proof}
The bounds on $n$ are tight in the corollary.  
For $\ell=b=1$, $\CC_4^{(\s)}$ is Gorenstein, with Gorenstein point $c=(1,2,7,12)$.
For $\ell=5=-b$, $\CC_5^{(\s)}$ is Gorenstein, with Gorenstein point $c=(1,6,25,94,345)$.



\section{Ehrhart theory and symmetric $h^*$-vectors}\label{ehrhartsection}

Another consequence of our main result involves the $h^*$-vector of the
polytope $\RR_n^{(\s)}$ associated with the $\s$-lecture hall partitions, introduced in Section~\ref{overviewgorcone}.

Let 
\[
 i \left( \RR_n^{(\s)}, t \right) \ = \left | \left\{ \la \in \LL_n^{(\s)}  : \, \la_n \leq t \right\} \right | .
\]
The function $i \left( \RR_n^{(\s)}, t \right)$ is known to be a {\em quasi-polynomial} in $t$, i.e., it has the form $a_n(t) \, t^n + a_{ n-1 }(t) t^{ n-1 } + \dots + a_0(t)$, where $a_0(t), a_1(t), \dots, a_n(t)$ are periodic functions of $t$; the lcm of their periods is the \emph{period} of $\RR_n^{(\s)}$
 \cite{ccd,ehrhartpolynomial}.
The {\em Ehrhart series} of $\RR_n^{(\s)}$ is the series
\[
{\mathcal{E}}_n^{(\s)}(x) = 
\sum_{t \geq 0} i \left( \RR_n^{(\s)}, t \right) x^t.
\]
It was shown in \cite{PS} that
\begin{equation*}
{\mathcal{E}}_n^{(\s)}(x)\ = \ 
\frac{Q_n^{(\s)}(x)}{(1-x^{s_n})^{n+1}} \, ,
\end{equation*}
where
$Q_n^{(\s)}(x)$
is a polynomial with positive integer coefficients which can be interpreted in terms of statistics on ``$\s$-inversion sequences''.
Furthermore, from \cite[eq. (17)]{PS}, $Q_n^{(\s)}(1)= s_n(s_1s_2 \cdots s_n)$.
The coefficient sequence of $Q_n^{(\s)}(x)$ is referred to as the
 $h^*$-vector of the polytope $\RR_n^{(\s)}$.  For example,
\[
{\mathcal{E}}_n^{(1,3,5)}(x) = \frac{1+ 2x+ 4x^2+ 6x^3+ 9x^4+ 10x^5+ 11x^6+ 10x^7+ 9x^8+
6x^9+ 4x^{10}+ 2x^{11}+ x^{12}}{(1-x^5)^4} \, ,
\]
so the  $h^*$-vector of the polytope $\RR_n^{(1,3,5)}$ is $[1, 2, 4, 6, 9, 10, 11, 10, 9, 6, 4, 2, 1]$.
(Actually, $Q_n^{(\s)}(x)$ and $(1-x^{s_n})^{n+1}$  always have common factors \cite[Corollary~3]{PS}, but we won't make use of this here.)

Which sequences $\s$ give rise to symmetric $h^*$-vectors?
Conjectures were made in \cite{PS} which we will resolve here.  


\begin{theorem} \label{hvector}
 For a sequence $\s$ of positive integers,
the $h^*$-vector of $\RR_n^{(\s)}$ is symmetric
if and only if 
  $\CC_n^{(\s)}$ is Gorenstein. 
\end{theorem}
\begin{proof}
Note that the $h^*$-vector of $\RR_n^{(\s)}$ is symmetric if and only if $Q_n^{(\s)}(x)$ is self-reciprocal, which occurs if and only if
$\mathcal{E}_n^{(\s)}(x)$ is self-reciprocal.

The weight  function $w(\lambda)= \lambda_n$ satisfies the conditions of Theorem~\ref{thm:StanGor}, so $\CC_n^{(\s)}$ is Gorenstein if and only if
the function $f(x)$, defined by
\[
f(x) = \sum_{\la \in \LL_n^{(\s)}} x^{\la_n},
\]
is self reciprocal.  To relate $f(x)$ to $\mathcal{E}_n^{(\s)}(x)$,
\begin{align*}
f(x) & = \sum_{\la \in \LL_n^{(\s)}} x^{\la_n}
  = \sum_{t \geq 0}  \sum_{\la \in \LL_n^{(\s)}; \, \la_n=t} x^{t}\\
& = 1 +  \sum_{t \geq 1} x^t   \left (i \left( \RR_n^{(\s)}, t \right)- i \left( \RR_n^{(\s)}, t-1 \right)\right)\\
& = (1-x)  \sum_{t \geq 0} x^t i \left( \RR_n^{(\s)}, t \right).
\end{align*}
Thus, $\mathcal{E}_n^{(\s)}(x)= f(x)/(1-x)$, which is self-reciprocal if and only if $f(x)$ is.
\end{proof}

\begin{remark}
For the geometrically inclined, Theorem~\ref{hvector} also follows from the fact that $\CC_n^{(\s)}$ is Gorenstein if and only if the cone over $\RR_n^{(\s)}$ is Gorenstein (see \cite[Section 3.1]{ccd} for the definition of a cone over a polytope).
This follows easily from the observation that if $(c_1,\ldots,c_n)$ is the Gorenstein point for $\CC_n^{(\s)}$, then $(c_1,\ldots,c_n,c_n+1)$ is the Gorenstein point for the cone over $\RR_n^{(\s)}$. 
\end{remark}

Combining Theorems~\ref{hvector} and \ref{genBMEsr}, we have:

\begin{theorem}
Let $\s=(s_1, s_2, \ldots, s_n)$ be a sequence of positive integers such that $\gcd(s_i,s_{i+1})=1$
for $1 \leq i < n$.  Then the $h^*$-vector of $\RR_n^{(\s)}$ is symmetric if and only if
$\s$ is $\u$-generated by a sequence
of positive integers.
\label{genBMEhvector}
\end{theorem}
For example,  in Section 2.3, the ``$1 \bmod k$'' sequences are $\u$-generated by
$\u=(k+2,2,2,2,\ldots)$, so by Theorem~\ref{genBMEhvector}, the $h^*$-vector of $\RR_n^{(\s)}$ is symmetric,
confirming \cite[Conjecture 5.5]{PS}.

Combining Theorems~\ref{hvector} and \ref{ell-Gorenstein}, we have the following.

\begin{theorem} \label{ell-conjecture}
Let $\ell > 0$ and $b \not = 0$ be integers satisfying $\ell^2+4b \geq 0$.
Let $\s$ be defined by the recurrence
$$s_n = \ell s_{n-1} + b s_{n-2},$$
with initial conditions $s_1=1$, $s_0=0$. The  $h^*$-vector of $\RR_n^{(\s)}$ is symmetric
 for all $n \geq 0$ if and only if $b=-1$.
If $b \not = -1$, there is an integer $n_0 = n_0(b,\ell)$ so that for all $n \geq n_0$,
the $h^*$-vector of $\RR_n^{(\s)}$ is not symmetric.
\end{theorem}

The first conclusion in  Theorem~\ref{ell-conjecture} proves \cite[Conjectures 5.6 and 5.7]{PS}.

\section{Concluding remarks}





%
%
%

We summarize the main contributions of this paper and offer two related problems.

For partition theory,  our main contibution is the negative result that among the sequences $\s$ defined by $s_n=\ell s_{n-1} + bs_{n-1}$ with initial conditions $s_0=0,s_1=1$, if $\s$ is not an $\ell$-sequence, the generating function for the $\s$-lecture hall
partitions can be of the form $((1-q^{e_1}) \cdots (1-q^{e_n}))^{-1}$ for at
most finitely many $n$.

We proved a geometric result showing that $\ell$-sequences are unique among sequences defined by second-order linear recurrences.
Before this, we knew special things about $\ell$-sequences:  they generalized the integers ($\ell=2$);
they  give rise to an $\ell$-generalization of Euler's Partition Theorem \cite{BME2,SY};  they give rise to an $\ell$-nomial coefficient which has a $q$-analogue with a meaningful interpretation \cite{LS}.
But until now, we did not have negative results to distinguish $\ell$-sequences from other second order linear recurrences.

The theory of Gorenstein cones provided a direct translation of results about lecture hall partitions into results about lecture hall cones.  
Our results provide many new examples of Gorenstein cones, in particular,  the $\s$-lecture hall cones where $\gcd(s_{i+1},s_1)=1$ and $\s$ is $\u$-generated.
In Ehrhart theory, we proved that the $h^*$-vector of the rational $\s$-lecture hall polytope is symmetric when $\s$ is an $\ell$-sequence.
And, except for finitely many $n$, it is not for any other second order sequence.



For $\ell$-sequences, the Gorenstein point is
\[ (s_1,s_1+s_2,s_2+s_3,\ldots,s_{n-1}+s_n) \]
while the generating function for $\s$-lecture hall partitions is
\[
\sum_{\lambda\in L_n^{(\s)}}q^{|\lambda|}=\frac{1}{(1-q^{s_1})(1-q^{s_1+s_2})\cdots (1-q^{s_{n-1}+s_n})}.
\]
Can we explain why the coordinates of the Gorenstein point and the  exponents in the generating function are the same?
Is there an explanation at all?  Or is this just coincidence?

Is it true that for every sequence $\s$ of positive integers, the $h^*$-vector of the rational $s$-lecture hall polytope is unimodal (as, e.g., for the ``1 mod $k$" sequences)?
We have tested many sequences $\s$, including ones that were not themselves monotone and found no counterexamples.
The polynomials $Q_n^{(\s)}(x)$   do not, in general, have all roots real.


\bibliographystyle{plain}
\bibliography{square}

\end{document}